\documentclass[a4paper,12pt]{amsart}
\usepackage[utf8]{inputenc}
\usepackage{amssymb, mathrsfs, amsfonts, amsmath}
\usepackage{mathtools} 
\usepackage{hyperref}
\usepackage{amsthm}
\usepackage{tikz}
\usepackage[english]{babel}
\usepackage{xcolor} 
\usepackage{geometry}
\title[Maximal Restriction Estimates]{Maximal restriction estimates and the maximal function of the Fourier transform}
\author{Jo\~ao P. G. Ramos}

\newcommand{\R}{\mathbb{R}}

\newcommand{\C}{\mathbb{C}}
\newcommand{\mmd}{\mathrm{d}}

\newtheorem{theorem}{Theorem}

\newtheorem{lemma}{Lemma}

\def\Xint#1{\mathchoice
{\XXint\displaystyle\textstyle{#1}}%
{\XXint\textstyle\scriptstyle{#1}}%
{\XXint\scriptstyle\scriptscriptstyle{#1}}%
{\XXint\scriptscriptstyle\scriptscriptstyle{#1}}%
\!\int}
\def\XXint#1#2#3{{\setbox0=\hbox{$#1{#2#3}{\int}$}
\vcenter{\hbox{$#2#3$}}\kern-.5\wd0}}

\def\dashint{\Xint-}
\begin{document}
\begin{abstract}
We prove a maximal restriction inequality for the Fourier transform, providing an answer to a question left open by M\"uller, Ricci and Wright \cite{MRW}. Our methods are similar to the ones in \cite{MRW} and \cite{CarlSjoe}, with the addition 
of a suitable trick to help us linearise our maximal function. In the end, we comment on how to use the same trick in combination with Vitturi's approach \cite{Vitturi} to obtain a stronger high-dimensional result.  
\end{abstract}
\maketitle

\section{Introduction} 

Restriction estimates for the Fourier transform have been a very active topic within harmonic analysis for over the past 40 years. Basically, one inquires whether an inequality of the form 

\begin{equation}\label{restrict} 
\|\widehat{f}|_{S}\|_{L^q(S,\mmd\sigma)} \le C_{p,d} \|f\|_{L^p(\R^d)}
\end{equation}
can hold on a hypersurface $S$, where $\sigma$ stands for the standard surface measure on $S$, which is the same as the arclength measure for the case of plane curves. Here we shall focus on compact hypersurfaces $S$ with non-vanishing curvature, the typical example being the sphere $\mathbb{S}^{d-1}.$ 
By taking examples of functions (either the so called \emph{Knapp} examples or constant functions; see, e.g., \cite[Section~4]{Tao}), one finds out that a \emph{necessary} condition for such inequalities to hold is that 
\begin{equation}\label{range}
1 \le p < \frac{2d}{d+1} \text{ and } p' \ge \frac{d+1}{d-1}q, 
\end{equation} 
where $\frac{1}{p} + \frac{1}{p'} = 1.$ The \emph{restriction conjecture} then asserts that the above conditions are also \emph{sufficient}. \\ 

The first manifestation of such a restriction principle, in a range smaller than \eqref{range}, was perhaps the result of Fefferman and Stein (see \cite[page~28]{Fefferman}), where 
an estimate in all dimensions for $q=2$ was proven, this estimate being sharpened to the optimal range of $p$ for such $q$ by Tomas \cite{Tomas}, who credits Stein for the endpoint result. 
For the sphere (and, in general, for compact hypersurfaces with non-vanishing curvature), it reads that
\[
\|\widehat{f}|_{S}\|_{L^2(S,d\sigma)} \le C_{p,d} \|f\|_{L^p(\R^d)},
\]
whenever $1 \le p \le \frac{2(d+1)}{d+3}.$ \\ 

Regarding ranges of exponents, for dimension $d \ge 3,$ Problem \eqref{restrict} is still open, with new technology being developped continously to improve ranges of exponents; see, for instance, \cite{Tao, Guth1, Guth2} for further developments in this subject. \\ 

For dimension 2, however, Problem \eqref{restrict} has been completely solved, as we observe that the conditions can be rewritten as follows: 
\begin{equation}\label{ranger}
1\le p < \frac{4}{3}, \, p' \ge 3q. 
\end{equation}
In the non-endpoint case $ p' > 3q$, the result is due to Fefferman \cite[page~33]{Fefferman}, and the endpoint to Zygmund \cite{Zygmund} and Carleson and Sj\"olin \cite{CarlSjoe}. Later, Sj\"olin \cite{Sjoelin} also extended these 
results to other classes of curves. \\

In \cite{MRW}, D. M\"uller, F. Ricci and J. Wright consider a slight strenghtening of the restriction properties of the Fourier transform in two dimensions: namely, they prove a maximal version of restriction estimates and conclude a differentiation result. Here, we shall state the result only in the case of $\mathbb{S}^1,$ for simplicity: 

\begin{theorem}\label{teoMRW}[M\"uller, Ricci, Wright \cite{MRW};\,2016] Let $\mathbb{S}^1$ be the unit circle in $\R^2$ and $f:\R^2 \to \C$ be a $L^p$ function. Assume that $1 \le p < \frac{8}{7}.$
Then, with respect to arclength measure, almost every point in $\mathbb{S}^1$ is a Lebesgue point for $\widehat{f}$ and the regularised value of $\widehat{f}$ at $x$ coincides with the restriction operator $\mathcal{R}f(x)$ for almost every $x \in \mathbb{S}^1.$ 
\end{theorem}

The purpose of this note is to improve ranges of exponents of such maximal restriction results. Explicitly, our main result is:

\begin{theorem}\label{improvement} Theorem \ref{teoMRW} extends to $1\le p < \frac{4}{3}.$
\end{theorem}

The strategy in \cite{MRW} passes through a maximal function with absolute values \emph{outside} the integral, and then uses H\"older inequality. Namely, they focus on maximal functions of the form 
\[ 
\mathcal{M}f(x) = \sup_{\substack{R \text{ axis parallel}, \\ \text{centered at} x}} \left| \int \chi_R(y) \widehat{f}(y) \,\mmd y\right|,
\]
where $\chi_R \in \mathcal{S}(\R)$ is a smooth bump function adapted to the rectangle $R$. They then prove that, for the \emph{whole} restriction range $1 \le p < \frac{4}{3}$ and $p' \ge 3q,$ 
\[ 
\|\mathcal{M}f\|_{L^q(\mmd \sigma)} \le C_{p,\Gamma}\|f\|_{L^p(\R^2)},
\]
where $\sigma$ stands again for the arclength measure on the curve $\Gamma$. Finally, in order to prove Theorem \ref{teoMRW}, the authors bound the maximal function 
\begin{equation}\label{grandfourier} 
M_{\mathcal{R}}f(t) = \sup_{\substack{R \text{ axis parallel}, \\ \text{centered at} x}} \int \chi_R(y) |\widehat{f}(y)| \, \mmd y
\end{equation}
by $(\mathcal{M}h(x))^{1/2}$, where $h = f * \tilde{f},$ with $\tilde{f}(x,y) = f(-x,-y).$ \\ 

In order to prove Theorem \ref{improvement}, it suffices to bound \eqref{grandfourier} from $L^r(\R^2)$ to some $L^q(\mathbb{S}^1),$ where $1 \le r < \frac{4}{3},$ as the stated property holds trivially in the class $\mathcal{S}(\R^2).$ By nature of such an approximation argument, it sufficies to prove these bound for functions $f \in \mathcal{S}(\R^2)$. \\

For fixed $g$ with $\|g\|_\infty=1$ and measurable choice $R$ of axis-parallel rectangles, define the linearised maximal operator
\begin{equation}\label{linearr}
 M_{g,R}f(x)=\int_{\R^2} |R(x)|^{-1} 1_{R(x)}(y-x)\widehat{f}(y) g(y)\, \mmd y
\end{equation}
acting initially, say, on functions in $L^1(\R^2)\cap L^2(\R^2)$. Setting $g(y) = \frac{\overline{\widehat{f}}}{|\widehat{f}|}$
where $\widehat{f} \ne 0,$ and zero otherwise, together with measurable $R$ such that $M_{g,R}f(t) \ge \frac{1}{2} M_{\mathcal{R}}f,$ implies that it is in turn sufficient to estimate \eqref{linearr} from $L^r(\R^2)$ to some $L^q(\mathbb{S}^1).$ This is the basic goal of Lemmata \ref{lemma1} and \ref{lemma2}. \\ 

Following \cite{MRW}, M. Vitturi \cite{Vitturi} and V. Kovac and D. Oliveira e Silva \cite{KovacOliveira} have proved, as a consequence of $p'=4$  being admissible for the restriction estimate, results in dimensions $\ge 3$
: they have obtained that, in the same range of exponents as in Theorem \ref{dim3more}, one gets \emph{pointwise convergence} $\chi_{\varepsilon} * \widehat{f} \to \widehat{f}$ for $\sigma-$almost every point on the sphere $\mathbb{S}^{d-1},$ where $\chi_{\varepsilon}(y) = \frac{1}{\varepsilon^n} \chi(y/\varepsilon),$ 
and $\chi \in \mathcal{S}(\R^d).$ Although this is already present in \cite{Vitturi} and in both cases the techniques also imply the same result for $\chi = 1_{B(0,1)},$ the ideas in \cite{KovacOliveira} represent a stronger, quantitative form of such a theorem, as they consider \emph{variation norms} instead of suprema. \\ 

Our second result is also an improvement on Vitturi's techniques, yet in another direction:

\begin{theorem}\label{VitturiImp} Let $f \in L^p(\R^d), \, 1 \le p \le \frac{4}{3}.$ Then $\sigma-$almost every point of $\mathbb{S}^{d-1}$ is a Lebesgue point of $\widehat{f},$ and the regularised value of $\widehat{f}$ at $x$ coincides with the restriction operator
$\mathcal{R}f(x)$ for almost every $x \in \mathbb{S}^{d-1}.$ 
\end{theorem}

The argument to prove Theorem \ref{VitturiImp} is similar to the one employed to treat Theorem \ref{improvement}, and we postpone it to the end of this manuscript. 

\section{Main Argument}
Call a measurable function $a$ in $\R^d$ \emph{bump function} if
there exists an axis parallel rectangle $R$ centered at the origin with
$$|a|\le |R|^{-1}1_R.$$
Convolution with such a bump
function satisfies a pointwise bound by the strong Hardy Littlewood maximal function, uniformly in the rectangle. The following lemma
concerns an adjoint of a linearised maximal operator, 
combined with a Fourier transform. 

\begin{lemma}\label{lemma1}
For each $x \in \R^d$ let $a_x$ be a convolution product of $k$ bump functions. Assume further that ${a_x}(y)$, as function in $(x,y)$, is in $ L^\infty(x, L^1(y))$. Let  $T$ be defined on functions 
$f\in L^2(\R^d)\cap L^1(\R^d)$ by
$$Tf(\xi)=\int_{\R^d} \widehat{a}_x(\xi) e^{2\pi i x  \cdot \xi}  f(x)\, \mmd x .$$
Then, for some universal constant $C$ depending on $k$ and $d$ only, 
$$\|Tf\|_2\le C \|f\|_2 .$$
\end{lemma}

\begin{proof}
We set up a duality argument, testing $Tf$ against  an arbitrary function $ \widehat{g}\in L^2(\R^d)\cap L^1(\R^d)$.
We have, by Fubini and Plancherel,
$$ \int_{\R^d}\overline{\widehat{g}(\xi)} \int_{\R^d}  \widehat{a}_x(\xi) e^{2\pi i x  \cdot \xi}  f(x)    \, \mmd x   \mmd\xi=
\int_{\R^d} f(x) \int_{\R^d} \overline{{g}(y)}  {a}_x(y-x) \, \mmd y \mmd x$$
Identifying on the right-hand-side a $k$ fold convolution by bump functions acting on $g$, we estimate the last display by
$$\int_{\R^d}\int_{\R^d}     |f(x)| M^k({g})(x)\, \mmd x
\le \|f\|_2 \|M^k g\|_2\le  C\|f\|_2\|\widehat{g}\|_2\ ,$$
where we have used the strong maximal theorem and Plancherel again.
 Since $\widehat{g}$ was arbitrary, this proves Lemma \ref{lemma1}.

 \end{proof}

The hypotheses  in the next Lemma  
are motivated by the parameterised circle
$$z(t)=(\cos(2\pi t), \sin(2\pi t)) .$$
By the addition theorem for the sine function, we have
$$|\det(z'(t),z'(s))|=4\pi^2 |\sin(2\pi (s-t))| .$$
Note the vanishing of the determinant when the two tangent vectors become parallel
or anti-parallel. Note further that one can recover $t\neq s\in I:=[0,1)$ from  
$$x:=z(t)+z(s).$$
Namely, $x/2$ is the midpoint between $z(t)$ and $z(s)$, and these two points 
on the circle are mirror symmetric relative to the line through this
midpoint and the origin. This determines the two points $t\ne s,$ up to
permutation. Define, therefore, the \emph{upper triangle}
$$\Delta=\{(t,s)\in I\times I: t >s\} .$$

 \begin{lemma}\label{lemma2}
Let $z: \R\to \R^2$ be a 
smooth one-periodic curve such that for all $(t,s)\in \Delta$
\begin{equation}
\label{sine}
|\det(z'(t),z'(s))|\ge c|\sin(2\pi(t-s))|
\end{equation}
and such that the map 
\begin{equation}
\label{bijection}(t,s)\to z(t)+z(s)
\end{equation}  
is a bijection from $\Delta$ onto a bounded set $\Omega\subset \R^2$.
 With $a_{z(t)}$ a bump function for every $t\in I$ such that 
$a_{z(t)}(x)$ is in $L^{\infty}(t,L^1(x))$, consider an operator
acting on functions in $L^4(I)$ as follows:
$$Tf(\xi)=\int_{I} \widehat{a}_{z(t)}(\xi) e^{2\pi  \xi \cdot z(t)}  f(t  )\, \mmd t .$$ 
Then we have for all $1\le  p< 2$ with some constant depending only on $p$:
$$\|Tf\|_{2p'}\le C_p \|f\|_{\frac{2p}{3-p}} ,$$
with the obvious interpretation when $p'=\infty$. Notice, moreover, that the reciprocals $\left(\frac{1}{2p'},\frac{3-p}{2p}\right)$ of the aforementioned exponents 
lie on the line segment joining $(1/4,1/4)$ and $(0,1).$
\end{lemma}

\begin{proof}
To reduce to Lemma \ref{lemma1}, we need to pass to a two dimensional integral. We follow the idea of Carleson-Sj\"olin and consider the square 
 $$Tf(\xi)^2=\int_{I\times I} \widehat{a}_{z(t)}(\xi) \widehat{a}_{z(s)}(\xi)e^{2\pi  \xi \cdot (z(t)+z(s))}  f(t ) {f(s)}\, \mmd t \mmd s . $$
The integral is twice the  analoguous integral over
the triangle $\Delta$, where we change coordinates by the bijective map \eqref{bijection} to obtain

$$Tf(\xi)^2= 2 \int_\Omega  \widehat{b}_{x}(\xi)e^{2\pi i \xi \cdot  x}  g(x)\, \mmd x  .$$
Here  we have unambiguously defined, for $(t,s)$ in the triangle,
$$\widehat{b}_{z(t)+z(s)}:=\widehat{a}_{z(t)}   {\widehat{a}_{z(s)}} ,$$
$$g(z(t)+z(s)):=f(t)  {f(s)}|\det(z'(t),z'(s))|^{-1} .$$
Note that the determinant here is the Jacobian determinant of the
map \eqref{bijection}.

It is now easy to prove, by interpolation, that for $1\le p\le 2$
we have   
$$\|Tf\|_{2p'}^{2p}=\|(Tf)^2\|_{p'}^p\le C\|g\|_p^p\ .$$
Namely,  $p=2$ follows directly from Lemma \ref{lemma1} applied to a function supported on $\Omega$, and $p=1$ is trivial since  $\|\widehat{b}_x\|_\infty\le C$.
To conclude the proof of the lemma, we invert  
the change of variables to estimate the right-hand-side for $1\le p<2:$ 
$$\int_{\Omega} |g(x)|^p \, \mmd x= \int_{\Delta}    |f(t  )f(s)|^{p} 
|\det(z'(t),z'(s))|^{1-p}
\, \mmd t \mmd s   \le   C_p\||f|^p\|^{2}_{\frac{2}{3-p}} = C_p \|f\|_{\frac{2p}{3-p}}^{2p}.$$
Here, the last inequality follows from the Hardy--Littlewood--Sobolev inequality for fractional integrals. Namely, we estimate with \eqref{sine} on the triangle:
$$|\det(z'(t),z'(s))|^{1-p} \le C \sum_{k=-2}^{2} |t-s-k|^{1-p} , $$ 
and we note that each summand leads to a translated fractional integral.
\end{proof}

\begin{proof}[Proof of Theorem \ref{improvement}] 
We introduce the bump function 
$$a_x(y):=|R(x)|^{-1} 1_{R(x)}(y)\overline{g(x-y)},$$
and write
$$M_{g,R}f(t)=\int_{\R^2} \overline{a_{z(t)}} (y-z(t))\widehat{f}(y) \, \mmd y.$$ 
This is just a composition of the operator in \eqref{linearr} with a parametrisation, so we identify them. By Plancherel, similarly to the proof of Lemma \ref{lemma1}, we have
$$M_{g,R}f(t)=\int_{\R^2} \overline{\widehat{a}_{z(t)}} (\xi) e^{-2\pi i \xi \cdot z(t)}f(\xi) \, \mmd y.$$ 
The adjoint operator then becomes
 $$M_{g,R}^*(h)(\xi)=\int_{I} \widehat{a}_{z(t)} (\xi)e^{2\pi i \xi \cdot z(t)}h(t) \, \mmd t.$$ 
 By Lemma \ref{lemma2}, this is bounded from $L^{\frac{2p}{3-p}}$ to $L^{2p'}$
 for $p<2$. We set now $r = (2p')'.$ By a computation, $\frac{2p}{3-p} = (r'/3)'.$ With this notation, we have that $M_{g,R}$ is bounded from $L^r(\R^2)$ to $L^{r'/3}(\mathbb{S}^1)$
 for all $r < \frac{4}{3},$ which is already what we wished to prove. Recall, moreover, that this implies $L^r(\R^2) \to L^q(\mathbb{S}^1)$ estimates in the optimal two-dimensional restriction range 
 $1 \le r < \frac{4}{3}, r' \ge 3q.$

\end{proof}

\section{The high-dimensional result} 

Just like we employed our techniques to deal with the two-dimensional case, we adapt the arguments by M. Vitturi \cite{Vitturi} to achieve high-dimensional estimates. We briefly sketch on how to do it.

\begin{theorem}\label{dim3more} Let $d \ge 3,$ and $$ \mathfrak{M}f(x) = \sup_{0 < \varepsilon \le 1} \dashint_{B(0,\varepsilon)} |\widehat{f}(x + y)| \mmd y.$$  
Then it holds that 

\[
\| \mathfrak{M}f\|_{L^q(\mathbb{S}^{d-1})} \le C_{p,q,d} \|f\|_{L^p(\R^{d})},
\]
where $1 \le p \le \frac{4}{3}$ and $p' \ge \frac{d+1}{d-1}q.$ 
\end{theorem}

\begin{proof} 
First, write the auxiliary bilinear operator 
\[
\mathfrak{M}(f;g)(x) = \sup_{0 \le \varepsilon \le 1} \left| \dashint_{B(0,\varepsilon)} \widehat{f}(x+y) g(x+y) \, \mmd y\right|. 
\]
Letting $\mathfrak{A}_{\varepsilon(\cdot),g}f(x) = \dashint_{B(0,\varepsilon(x))}  \widehat{f}(x+y) g(x+y) \,\mmd y$ be the linearised operator for suitable measurable $g,\varepsilon,$ $\|g\|_{\infty} =1$. Its adjoint has the form 
\[
\mathfrak{A}^*_{\varepsilon(\cdot),g}h(\xi) = \int_{\mathbb{S}^{d-1}} G(x,\xi) e^{2 \pi i \xi \cdot x} h(x)\, \mmd \sigma(x),
\]
$\sigma$ standing for the surface measure on the $(d-1)-$dimensional sphere, and $G(x,\xi) = \mathcal{F}(g(x + \cdot)\chi_{B(0,\varepsilon(x))})(\xi).$ Following Vitturi's arguments and the ones in the proof of Theorem \ref{improvement},
it is enough to prove the following estimate: 
\[
\|\mathfrak{A}^*_{\varepsilon(\cdot),g}h\|_{L^4(\R^d)} \le C_{q,d} \| h\|_{L^{q_d'}(\mathbb{S}^{d-1})},
\]
where $q_d = 4\frac{d+1}{d-1}.$ Now we write the $L^4$ norm as a (square root of a) $L^2$ norm of the convolution of the Fourier transform $(\mathfrak{A}^*_{\varepsilon(\cdot),g}h)^{\widehat{ }}$ with itself. With this in mind, one gets from a calculation that 
\[
(\mathfrak{A}^*_{\varepsilon(\cdot),g}h)^{\widehat{ }}(\eta) = g(\eta) \int_{\mathbb{S}^{d-1}} h(x) \chi_{B(0,\varepsilon(x))}(\eta-x) \mmd x =: g(\eta) T_{\varepsilon(\cdot)}h(\eta).
\]
We are then able to bound 
\[ 
|(\mathfrak{A}^*_{\varepsilon(\cdot),g}h)^{\widehat{ }} * (\mathfrak{A}^*_{\varepsilon(\cdot),g}h)^{\widehat{ }} (\rho)| \le |( T_{\varepsilon(\cdot)}|h|)*( T_{\varepsilon(\cdot)}|h|)(\rho)|.
\]
But the operator on the right hand side has been already treated in Vitturi's proof, and therefore we can conclude the desired bounds from the ones in \cite{Vitturi}. 
\end{proof} 

\section{Acknowledgements} The author is deeply indebted to his doctoral advisor, Prof. Dr. Christoph Thiele, for the time devoted to thorough inspection of the arguments, careful understanding of ideas and several suggestions that led to the final 
version of this manuscript. He would also like to akcnowledge the interest in this project manifested by some of his colleagues, among which he would like to thank especially Marco Fraccaroli and Gennady Uraltsev for discussions about their work \cite{FraccUraltsev}, as well as Diogo Oliveira e Silva and Vjekoslav Kovac,
which, besides sharing their recent project \cite{KovacOliveira}, also took the time carefully read and make valuable suggestions on an early version of this work.

\end{document}